\documentclass[a4paper]{article}

\usepackage{bm}
\usepackage{amsmath,amssymb}
\usepackage{amscd}
\usepackage{mathrsfs}
\usepackage{amsthm}
\usepackage[all]{xy}
\usepackage{color}

\theoremstyle{plain}

\newtheorem{theorem}{\bf Theorem}[section]

\theoremstyle{definition}

\newtheorem{example}[theorem]{\bf Example}

\newcommand{\eqa}[1]{
\begin{align*}
#1
\end{align*}}

\newcommand{\nai}[2]{\langle #1,#2\rangle}

  \makeatletter
  \newcommand{\subsubsubsection}{\@startsection{paragraph}{4}{\z@}%
    {1.0\Cvs \@plus.5\Cdp \@minus.2\Cdp}%
    {.1\Cvs \@plus.3\Cdp}%
    {\reset@font\sffamily\normalsize}
  }
  \makeatother
\title{When does the Weyl--von Neumann Theorem hold?}

\author{Hiroshi Ando\footnote{HA is supported by JSPS KAKENHI 16K17608 and Grant for Basic Science Research Projects from The Sumitomo
Foundation}\and Yasumichi Matsuzawa\footnote{YM is supported by JSPS KAKENHI 6800055 and 26350231}}

\begin{document}

\maketitle
\begin{abstract} 
A famous theorem due to Weyl and von Neumann asserts that two bounded self-adjoint operators are unitarily equivalent modulo the compacts, if and only if their essential spectrum agree. The above theorem does not hold for unbounded operators. Nevertheless, there exist closed subsets $M$ of $\mathbb{R}$ on which the Weyl--von Neumann Theorem hold: all (not necessarily bounded) self-adjoint operators with essential spectrum $M$ are unitarily equivalent modulo the compacts. In this paper, we determine exactly which $M$ satisfies this property. 
\end{abstract}
\section{Introduction and Main Theorem} 
Let $H$ be a separable infinite-dimensional complex Hilbert space, and let ${\rm{SA}}(H)$  (resp. $\mathbb{B}(H)_{\rm{sa}}$) be the set of all self-adjoint (resp. bounded self-adjoint) operators on $H$. Also let $\mathcal{U}(H)$ (resp. $\mathbb{K}(H)_{\rm{sa}}$) be the group of unitaries (resp. compact self-adjoint operators) on $H$. The essential spectrum of $A\in {\rm{SA}}(H)$ is denoted by $\sigma_{\rm{ess}}(A)$. 
The celebrated Weyl--von Neumann Theorem \cite{Weyl,Weyl10,von Neumann} asserts that operators $A,B\in \mathbb{B}(H)_{\rm{sa}}$ are unitarily equivalent modulo the compacts (which we call {\it Weyl--von Neumann equivalent}), that is, $uAu^*+K=B$ for some $u\in \mathcal{U}(H)$ and $K\in \mathbb{K}(H)_{\rm{sa}}$ if and only if $\sigma_{\rm{ess}}(A)=\sigma_{\rm{ess}}(B)$. This theorem has continued to play significant roles in many fields of analysis. On the other hand, for unbounded operators the Weyl--von Neumann Theorem does not hold (Weyl--von Neumann equivalent operators always have the same essential spectrum, but the converse need not hold). In fact, the Weyl--von Neumann equivalence, viewed as an equivalence relation on the Polish space ${\rm{SA}}(H)$ endowed with the strong resolvent topology is unclassifiable by countable structures, despite the fact that its restriction to $\mathbb{B}(H)_{\rm{sa}}$ is smooth \cite[Theorems 3.12 and 3.33]{AM15}. Somewhat unexpectedly, it is shown that the Weyl--von Neumann Theorem holds on $\mathbb{R}$ \cite[Theorem 3.17 (2)]{AM15}, i.e., any two operators $A,B\in {\rm{SA}}(H)$ with $\sigma_{\rm{ess}}(A)=\sigma_{\rm{ess}}(B)=\mathbb{R}$ are always Weyl--von Neumann equivalent, while it fails on $\emptyset$ or $\mathbb{N}$ \cite[Examples 3.3 and 3.5]{AM15}. It would be natural to ask on which closed subsets $M$ of $\mathbb{R}$ the Weyl--von Neumann Theorem holds. The answer to the question is the main result of the paper: 
\begin{theorem}
Let $M$ be a closed subset of $\mathbb{R}$. The following two conditions are equivalent.
\begin{list}{}{}
\item[{\rm{(i)}}] The Weyl-von Neumann Theorem holds on $M$. That is, 
$$\forall A,B\in {\rm{SA}}(H)\ [\sigma_{\rm{ess}}(A)=\sigma_{\rm{ess}}(B)=M\Rightarrow \exists u\in \mathcal{U}(H)\ \exists K\in \mathbb{K}(H)_{\rm{sa}}\ (uAu^*+K=B)].$$
\item[{\rm{(ii)}}] $M$ has no large holes at infinity. That is, $M\neq \emptyset$ and  
\[(*)\ \  d_M:=\lim_{n\to \infty}\min \left \{ \sup_{\lambda\in \mathbb{R}\setminus (M\cup [-n,n])}{\rm{dist}}(\lambda,M),1\right \}=0.\]
\end{list}
Here we assume $\sup \emptyset =0$. 
\end{theorem}
\begin{proof}
(ii)$\Rightarrow $(i): Assume that $M$ has no large holes at infinity, and let $A,B\in {\rm{SA}}(H)$ be such that $\sigma_{\rm{ess}}(A)=\sigma_{\rm{ess}}(B)=M$. 
By the Weyl's compact perturbation Theorem, we may assume that $A$ and $B$ are both diagonal with eigenvalues $\{\lambda_n\}_{n=1}^{\infty}, \{\mu_n\}_{n=1}^{\infty}\subset \mathbb{R}$, and moreover that all eigenvalues are of simple multiplicity. This implies that the sets of accumulation points of $\{\lambda_n\}_{n=1}^{\infty}$ and $\{\mu_n\}_{n=1}^{\infty}$ both equal $M$. 
Following von Neumann's proof (see \cite[$\S$94]{AkhiezerGlazman}),  let 
$$a_n:=\inf_{t\in M}|\lambda_n-t|, b_n:=\inf_{t\in M}|\mu_n-t|\ \ \ (n\in \mathbb{N}).$$
Then we show that $\lim_{n\to \infty}a_n=\lim_{n\to \infty}b_n=0.$ Assume by contradiction that $a_n$ does not converge to 0 as $n\to \infty$. 
Then there exist $0<\delta<1$ and a subsequence $(a_{n_k})_{k=1}^{\infty}$ such that $a_{n_k}\ge \delta\ (k\in \mathbb{N})$ holds. We first observe that for a fixed $N\in \mathbb{N}$, there exist only finitely many $k\in \mathbb{N}$ for which $|\lambda_{n_k}|\le N$. Indeed, assume by contradiction that there exists a subsequence $(n_k')_{k=1}^{\infty}$ of $(n_k)_{k=1}^{\infty}$ for which $|\lambda_{n_k'}|\le N\ (k\in \mathbb{N})$. Then $(\lambda_{n_k'})_{k=1}^{\infty}$ must have an accumulation point, say $\lambda$ with $|\lambda|\le N$. Then $\lambda\in M$ and there exists $k_0\in \mathbb{N}$ such that 
$\delta\le a_{n_{k_0}'}\le |\lambda_{n_{k_0}'}-\lambda|<\delta$,\,
which is a contradiction. Therefore by taking further subsequence of $(a_{n_k})_{k=1}^{\infty}$ if necessary, we may assume that $|\lambda_{n_k}|>k$ for every $k\in \mathbb{N}$. Then $\lambda_{n_k}\in \mathbb{R}\setminus (M\cup [-k,k])$. Let $0<\varepsilon<\delta(<1)$. By $(*)$, there exists $k_0\in \mathbb{N}$ such that 
for every $k\ge k_0$, and $\lambda \in \mathbb{R}\setminus (M\cup [-k,k]), \text{dist}(\lambda,M)<\varepsilon$ holds. 
This shows that $a_{n_{k_0}}=\text{dist}(\lambda_{n_{k_0}},M)<\delta$, a contradiction. Therefore $\lim_{n\to \infty}a_n=0$. Similarly, $\lim_{n\to \infty}b_n=0$ holds.
Then as in \cite[$\S$94]{AkhiezerGlazman}, there exists a bijection $\pi\colon \mathbb{N}\to \mathbb{N}$ such that $\lim_{n\to \infty}|\lambda_{\pi(n)}-\mu_n|=0$. By (the proof of) \cite[$\S94$, Theorem 3]{AkhiezerGlazman}, this shows that $A,B$ are Weyl-von Neumann equivalent.\\ \\
(i)$\Rightarrow $(ii): We show the contrapositive. Assume that (ii) does not hold. If $M=\emptyset$, then the Weyl-von Neumann Theorem does not hold on $M$ (what is much worse, the Weyl-von Neumann  equivalence relation restricted to $\{A\in {\rm{SA}}(H);\ \sigma_{\rm{ess}}(A)=\emptyset\}$ is still unclassifiable by countable structures \cite[Theorem 3.32]{AM15}). Hence we may assume that $M\neq \emptyset$ and $(1\ge )d_M>0$. Note that this in particular means that for each $n\in \mathbb{N}$, $\mathbb{R}\setminus (M\cup [-n,n])\neq \emptyset$ holds. Then at least one of $(0,\infty)\setminus M$ or $(-\infty,0)\setminus M$ is unbounded. 
We may therefore assume that $(0,\infty)\setminus M$ is unbounded, so that there exist numbers $1<\lambda_1<\lambda_2<\cdots$ in $\mathbb{R}\setminus M$ such that 
$\text{dist}(\lambda_n,M)>\frac{1}{2}d_M$ and $\lambda_{n+1}>2\lambda_n$ for every $n\in \mathbb{N}$. 
Since $M\neq \emptyset$, let $\{\mu_n\}_{n=1}^{\infty}$ be a countable dense subset of $M$ (it is possible that some $\mu_n$ and $\mu_m$ are equal for different $n,m$). Fix a bijection $\nai{\ \cdot\ }{\ \cdot\ }\colon \mathbb{N}^2\to \mathbb{N}$ given by $\nai{k}{m}=2^{k-1}(2m-1)\ (k,m\in \mathbb{N})$. Fix an orthonormal basis $\{\xi_n\}_{n=1}^{\infty}$ for $H$ and let $e_n$ be the orthogonal projection of $H$ onto $\mathbb{C}\xi_n\ (n\in \mathbb{N})$. Define $A,B\in {\rm{SA}}(H)$ by 
\begin{align}
A&:=\sum_{k=1}^{\infty}\lambda_ke_{\nai{k}{1}}+\sum_{k=1}^{\infty}\sum_{m=2}^{\infty}\mu_ke_{\nai{k}{m}},\\
B&:=\sum_{k=1}^{\infty}(\lambda_{k+1}-\tfrac{1}{4}d_M)e_{\nai{k}{1}}+\sum_{k=1}^{\infty}\sum_{m=2}^{\infty}\mu_ke_{\nai{k}{m}}.
\end{align}
Then 
$$\sigma_{\rm{ess}}(A)=\overline{\{\mu_n;n\in \mathbb{N}\}}=M=\sigma_{\rm{ess}}(B).$$  
We show that $A$ and $B$ are not Weyl-von Neumann equivalent. 
Assume by contradiction that there exist  $u\in \mathcal{U}(H)$ and $K\in \mathbb{K}(H)_{\rm{sa}}$ such that $uAu^*+K=B$ holds. Let $\eta_n:=u\xi_n\ (n\in \mathbb{N})$ and let $f_n=ue_nu^*$, the orthogonal projection of $H$ onto $\mathbb{C}\eta_n$. Then for each $k\in \mathbb{N}$, we have $uAu^*\eta_{\nai{k}{1}}+K\eta_{\nai{k}{1}}=B\eta_{\nai{k}{1}}$, so that 
\begin{equation}
K\eta_{\nai{k}{1}}=\sum_{k'=1}^{\infty}(\lambda_{k'+1}-\tfrac{1}{4}d_M-\lambda_k)\nai{\xi_{\nai{k'}{1}}}{\eta_{\nai{k}{1}}}\xi_{\nai{k'}{1}}+\sum_{k'=1}^{\infty}\sum_{m=2}^{\infty}(\mu_{k'}-\lambda_k)\nai{\xi_{\nai{k'}{m}}}{\eta_{\nai{k}{1}}}\xi_{\nai{k'}{m}}.
\end{equation}
Since $\nai{k}{1}\stackrel{k\to \infty}{\to}\infty$, we have $\eta_{\nai{k}{1}}\stackrel{k\to \infty}{\to}0$ weakly in $H$. Since $K$ is compact, this shows that $\|K\eta_{\nai{k}{1}}\|\stackrel{k\to \infty}{\to}0$. If $k=k'+1$, then $|\lambda_{k'+1}-\frac{1}{4}d_M-\lambda_k|=\frac{1}{4}d_M$. 
If $k<k'+1$, then 
$$\lambda_{k'+1}-\tfrac{1}{4}d_M-\lambda_k>2\lambda_k-\tfrac{1}{4}d_M-\lambda_k=\lambda_k-\tfrac{1}{4}d_M>\tfrac{1}{4}d_M.$$
If $k>k'+1$, then 
$$|\lambda_{k'+1}-\tfrac{1}{4}d_M-\lambda_k|>\lambda_{k'+1}+\tfrac{1}{4}d_M>\tfrac{1}{4}d_M.$$
This shows that in any case 
$$|\lambda_{k'+1}-\tfrac{1}{4}d_M-\lambda_k|\ge \tfrac{1}{4}d_M\ \ \ \ (k\in \mathbb{N}).$$
Also, for every $k'\in \mathbb{N}$, 
$$|\lambda_k-\mu_{k'}|\ge \text{dist}(\lambda_k,M)>\tfrac{1}{2}d_M>\tfrac{1}{4}d_M.$$
 Therefore for every $k\in \mathbb{N}$, we have  
\eqa{
\|K\eta_{\nai{k}{1}}\|^2&=
\sum_{k'=1}^{\infty}|\lambda_{k'+1}-\tfrac{1}{4}d_M-\lambda_k|^2|\nai{\xi_{\nai{k'}{1}}}{\eta_{\nai{k}{1}}}|^2+
\sum_{k'=1}^{\infty}\sum_{m=2}^{\infty}|\mu_{k'}-\lambda_k|^2|\nai{\xi_{\nai{k'}{m}}}{\eta_{\nai{k}{1}}}|^2\\
&\ge \sum_{k'=1}^{\infty}(\tfrac{1}{4}d_M)^2|\nai{\xi_{\nai{k'}{1}}}{\eta_{\nai{k}{1}}}|^2+
\sum_{k'=1}^{\infty}\sum_{m=2}^{\infty}(\tfrac{1}{4}d_M)^2|\nai{\xi_{\nai{k'}{m}}}{\eta_{\nai{k}{1}}}|^2\\
&=(\tfrac{1}{4}d_M)^2.
}  
This contradicts $\|K\eta_{\nai{k}{1}}\|\stackrel{k\to \infty}{\to}0$. Thus $A,B$ are not Weyl-von Neumann equivalent. 
\end{proof}
\begin{example}
Here are examples of closed sets $M\subset \mathbb{R}$ having no large halls at infinity. 
\begin{list}{}{}
\item[(a)] $M=\mathbb{R}$, or more generally $M=\mathbb{R}\setminus U$ with $U$ bounded and open. 
\item[(b)] $M=\mathbb{R}\setminus \bigcup_{n\in \mathbb{N}}(n-r_n,n+r_n)$, where $r_n>0\ (n\in \mathbb{N})$ and $\displaystyle \lim_{n\to \infty}r_n=0$.
\end{list}
\end{example}

Hiroshi Ando\\
Department of Mathematics and Informatics,\\
Chiba University\\
1-33 Yayoi-cho, Inage, Chiba, 263-8522
Japan\\
hiroando@math.s.chiba-u.ac.jp\\
\\
Yasumichi Matsuzawa\\
Department of Mathematics, Faculty of Education, Shinshu University\\
6-Ro, Nishi-nagano, Nagano, 380-8544, Japan\\
myasu@shinshu-u.ac.jp\\
\end{document}